\documentclass[a4paper,12pt]{article}
\usepackage{mathrsfs}
\usepackage{amssymb}
\usepackage{amsmath}
\usepackage{amsmath,amssymb,amsthm,graphics, graphicx,latexsym,amsfonts}
\usepackage{fancyhdr}
\usepackage{color}
\usepackage{bm}
\usepackage{tikz}
\usetikzlibrary{calc}
\usepackage{hyperref}
\usepackage{breakurl}

\title{Alon-Tarsi number of signed planar graphs\thanks{Supported by the
National Natural Science Foundation of China under Grant No.\,11471273 and 11561058.}}
\author
{
Wei Wang$^{\rm a,b}$,
Jianguo Qian$^{\rm a}$\thanks{Corresponding author: jgqian@xmu.edu.cn.}
\\
{\footnotesize$^{\rm a}$School of Mathematical Sciences, Xiamen University, Xiamen 361005, P. R. China}\\
{\footnotesize$^{\rm b}$College of Information Engineering, Tarim University, Alar 843300, P. R. China}
}
\usepackage{indentfirst}
\date{}
\begin{document}
\maketitle
\newtheorem{lem}{Lemma}[section]
\newtheorem{thm}[lem]{Theorem}
\newtheorem{prop}[lem]{Proposition}
\newtheorem{cor}[lem]{Corollary}
\newtheorem{conjecture}[lem]{Conjecture}
\newtheorem{defi}[lem]{Definition}
\newtheorem*{pf}{Proof}
\begin{abstract}
Let $(G,\sigma)$ be any signed planar graph. We show that  the Alon-Tarsi number of $(G,\sigma)$ is at most 5, generalizing a recent result of Zhu for unsigned case. In addition, if $(G,\sigma)$ is $2$-colorable then $(G,\sigma)$ has the Alon-Tarsi number at most 4. We also construct a signed planar graph which is $2$-colorable but  not $3$-choosable.
\end{abstract}

\section{Introduction}
Let $G$ be a simple graph with vertex set $V(G)$ and edge set $E(G)$. A \emph{signed graph} with underling graph $G$ is a pair $(G,\sigma)$, where $\sigma$ is a mapping from $E(G)$ to $\{+1,-1\}$. An edge $e$ is \emph{positive} (resp. \emph{negative}) if $\sigma(e)=+1$ (resp. $\sigma(e)=-1$). In particular, we denote by $(G,+)$ (resp. $(G,-)$) the signed graph $(G,\sigma)$ if every edge is positive (resp. negative).  We often identify  $(G,+)$ with the (unsigned) underling graph $G$.

Recently, based on the work of Zaslavsky \cite{Zaslavsky1982}, M\'{a}\v{c}ajov\'{a} et al. \cite{Raspaud2016} generalized the concept of chromatic number of an unsigned graph to a signed graph. For a signed graph $(G,\sigma)$ and a color set $C\subset \mathbb{Z}$, a \emph{proper coloring} \cite{Zaslavsky1982}  with color set $C$ is a mapping $\phi\colon\,V(G)\mapsto C$ such that
\begin{equation}\label{proper}
\phi(u)\neq \sigma(uv)\phi(v)
 \end{equation}
for each edge $uv\in E(G)$. For $k\ge 1$, set $M_k=\{\pm 1,\pm 2,\ldots,\pm k/2\}$ if $k$ is even and  $M_k=\{0,\pm 1,\pm 2\ldots,\pm (k-1)/2\}$ if $k$ is odd. A (proper) $k$-\emph{coloring} of a signed graph $(G,\sigma)$ is a proper coloring with color set $M_k$.  A signed graph $(G,\sigma)$ is $k$-\emph{colorable} if it admits a $k$-coloring. The \emph{chromatic number} of $(G,\sigma)$, denoted $\chi(G,\sigma)$, is the minimum $k$ for which $(G,\sigma)$ is $k$-colorable.

Jin et al. \cite{Jin2016} and Schweser et al. \cite{Schweser2017} further considered the list coloring of signed graphs.  For a positive integer $k$, a $k$-\emph{list assignment} of $(G,\sigma)$ is a mapping $L$ which assigns to each vertex $v$ a set $L(v)\subset \mathbb{Z}$ of $k$ permissible colors. For a $k$-list assignment $L$ of $(G,\sigma)$, an $L$-\emph{coloring} is a proper coloring $\phi\colon\,V(G)\mapsto\cup_{v\in V(G)}L(v)$ such that $\phi(v)\in L(v)$  for every vertex $v\in V(G)$. We say that $(G,\sigma)$ is $L$-\emph{colorable} if $G$ has an $L$-coloring. A signed graph $(G,\sigma)$ is called $k$-\emph{choosable} if  $G$ is $L$-colorable for any $k$-list assignment $L$. The \emph{list chromatic number} (or \emph{choice number}) $\chi_l(G,\sigma)$ is the minimum $k$ for which $G$ is $k$-choosable. Clearly, $\chi_l(G,\sigma)\ge \chi(G,\sigma)$. We note that when we restrict the signed graphs $(G,\sigma)$ to $(G,+)$, both the chromatic number and list chromatic number agree with the ordinary chromatic number and list chromatic number of its underlying graph $G$. This explains why we can identify $(G,+)$ with $G$.

Let `$<$' be an arbitrary fixed ordering of the vertices of $(G,\sigma)$. In view of (\ref{proper}),  we define the \emph{graph polynomial} of $(G,\sigma)$ as
 $$P_{G,\sigma}(\bm{x})=\prod_{u\sim v,u<v}(x_u-\sigma(uv)x_v),$$
 where $u\sim v$ means that $u$ and $v$ are adjacent, and $\bm{x}=(x_v)_{v\in V(G)}$ is a vector of $|V(G)|$ variables indexed by the vertices of $G$.  It is easy to see that a mapping $\phi\colon\,V(G)\mapsto \mathbb{Z}$ is a proper coloring of $(G,\sigma)$ if and only if $P_{G,\sigma}((\phi(v))_{v\in V(G)})\neq 0$.

 \begin{lem} \label{cnull}\cite{Alon1999}(Combinatorial Nullstellensatz) Let $\mathbb{F}$ be an arbitrary field and let $f=f(x_1,x_2,\ldots,x_n)$ be a polynomial in $\mathbb{F}[x_1,x_2,\ldots,x_n]$. Suppose that the degree $deg(f)$ of $f$ is $\sum_{i=1}^n t_i$ where each $t_i$ is a nonnegative integer, and suppose that the coefficient of $\prod_{i=1}^n x_i^{t_i}$ of $f$ is nonzero. Then if $S_1,S_2,\ldots,S_n$ are subsets of $\mathbb{F}$ with $|S_i|\ge t_i+1$, then there are $s_1\in S_1$,$s_2\in S_2$,\ldots,$s_n\in S_n$ so that $f(x_1,x_2,\ldots,x_n)\neq 0$.
 \end{lem}
 Note that $P_{G,\sigma}(\bm{x})$ is a homogeneous polynomial. It follows from Lemma \ref{cnull} that if there exists a monomial $c\prod_{v\in V(G)}x_v^{t_v}$ in the expansion of $P_{G,\sigma}(\bm{x})$ such that $c\neq 0$ and $t_v<k$ for all $v\in V(G)$, then $(G,\sigma)$ is $k$-choosable.
 Thus, the notion of Alon-Tarsi number of unsigned graphs defined by Jensen and Toft \cite{Jensen1995} can be naturally extended to signed graphs.
 \begin{defi}\label{AT}\textup{
 The \emph{Alon-Tarsi number} of $(G,\sigma)$, denoted $AT(G,\sigma)$, is the minimum $k$ for which  there exists a monomial $c\prod_{v\in V(G)}x_v^{t_v}$ in the expansion of $P_{G,\sigma}(\bm{x})$ such that $c\neq 0$ and $t_v<k$ for all $v\in V(G)$.}
 \end{defi}
Parallel to the unsigned case, we have
$$AT(G,\sigma)\ge \chi_l(G,\sigma)\ge \chi(G,\sigma).$$
 For a subgraph $H$ of $G$, we use $(H,\sigma)$ to denote the signed subgraph of $(G,\sigma)$ restricted on $H$, i.e.,  $(H,\sigma)=(H,\sigma|_{E(H)})$.  Note that $P_{H,\sigma}(\bm{x})$ is a factor of $P_{G,\sigma}(\bm{x})$. From Definition \ref{AT}, it is clear that $AT(H,\sigma)\le AT(G,\sigma)$.

 For a vertex $v$ in a signed graph $(G,\sigma)$, a \emph{switching} at $v$ means changing the sign of each edge incident to $v$. For $X\subseteq V(G)$, a switching at $X$ means switching at every vertex in $X$ one by one. Equivalently, a switching at $X$ means changing the sign of every edge with exactly one end in $X$. We denote the switched graph by $(G,\sigma^X)$. In particular,  when $X=\{v\}$ we use $(G,\sigma^v)$ to denote $(G,\sigma^{\{v\}})$. Two signed graphs $(G,\sigma)$ and $(G,\sigma')$ are \emph{switching equivalent} if $\sigma'=\sigma^X$ for some $X\subseteq V(G)$.

It is easy to show that two switching equivalent signed graphs have the same chromatic number \cite{Raspaud2016} as well as the same list chromatic number \cite{Jin2016,Schweser2017}. For the Alon-Tarsi numbers, we have the following similar result.
\begin{prop}\label{equAT}
 If two signed graphs $(G,\sigma)$ and $(G,\sigma')$ are switching equivalent then
$AT(G,\sigma)=AT(G,\sigma')$.
\end{prop}
\begin{proof}
It clearly suffices to consider the case that $\sigma'=\sigma^v$, where $v\in V(G)$. For any edge incident with $v$, say $uv$, we have $\sigma^v(uv)=-\sigma(uv)$.  We use $T(x_u,x_v)$ and $T^v(x_u,x_v)$ to denote the factors corresponding to this edge in $P_{G,\sigma}(\bm{x})$ and $P_{G,\sigma^v}(\bm{x})$, respectively. If $u<v$ then $T(x_u,x_v)=x_u-\sigma(uv)x_v$, $T^v(x_u,x_v)=x_u-\sigma^v(uv) x_v$ and hence $T(x_u,x_v)=T^v(x_u,-x_v)$. If $v<u$ then $T(x_u,x_v)=x_v-\sigma(uv)x_u$ and $T^v(x_u,x_v)=x_v-\sigma^v(uv) x_u$ and hence $T(x_u,x_v)=-T^v(x_u,-x_v)$. In either case we have $T(x_u,x_v)=\pm T^v(x_u,-x_v)$. Letting $\bm{x}^v$ be obtained from $\bm{x}$ by changing $x_v$ to $-x_v$, we have $P_{G,\sigma}(\bm{x})=\pm P_{G,\sigma^v}(\bm{x}^v)$. Therefore, for each monomial $\prod_{v\in V(G)} x_v^{t_v}$, the coefficients of this monomial in $P_{G,\sigma}(\bm{x})$ and $P_{G,\sigma^v}(\bm{x^v})$ and hence in  $P_{G,\sigma^v}(\bm{x})$ have the same absolute value. This implies that $AT(G,\sigma)=AT(G,\sigma^v)$.
\end{proof}
Recently, a few classical results on colorability \cite{Hu2018} and choosability \cite{Jin2016} of planar graphs were generalized to signed planar graphs. In particular, Jin et al. \cite{Jin2016} showed that every signed planar graph is 5-choosable, generalizing the well-known result of Thomassen \cite{Thomassen1994} which states that every (unsigned) planar graph is 5-choosable. Another generalization of  Thomassen's result was given by Zhu \cite{Zhu2018}, who showed that $AT(G)\le 5$ for any planar graph $G$, which solved an open problem proposed by Hefetz \cite{Hefetz2011}. Considering the above results of Jin et.~al \cite{Jin2016} and Zhu \cite{Zhu2018}, it is natural to ask whether the Alon-Tarsi number of each signed planar graph is at most 5. The main aim of this paper is to give an affirmative answer to this question.
\begin{thm}\label{AT5}
For any signed graph $(G,\sigma)$, if $G$ is a planar graph then $AT(G,\sigma)\le 5$.
\end{thm}
In \cite{Alon1992}, Alon and Tarsi showed that every bipartite planar graph is 3-choosable. The result is sharp as $K_{2,4}$ is a bipartite planar graph and $\chi_l(K_{2,4})=3$. The following result is a natural extension of this result for signed planar graphs.
\begin{thm}\label{AT4}
For any signed graph $(G,\sigma)$, if $G$ is planar and  $2$-colorable then $AT(G,\sigma)\le 4$. Moreover, there is a signed planar graph which is $2$-colorable but not $3$-choosable.
\end{thm}

\section{Orientation and Alon-Tarsi number for\\
signed graphs}
For an unsigned graph $G$, Alon and Tarsi \cite{Alon1992} found a useful combinatorial interpretation of the coefficient for each monomial in the graph polynomial $P_G(\bm{x})$ in terms of orientations and Eulerian subgraphs. By defining hypergraph polynomial and hypergraph orientation,  Ramamurthi and West \cite{Ramamurthi2005} generalized the result of Alon and Tarsi to $k$-uniform hypergraph for prime $k$. In this section we consider the signed graphs. Instead of using orientations of signed graphs, we use orientations of the underlying graphs and find that the result of Alon and Tarsi has a very natural extension for signed graphs.

Let $(G,\sigma)$ be a signed graph and `$<$' be an arbitrary fixed ordering of $V(G)$. For an orientation $D$ of the underling graph $G$, we denote by $(v,u)$ the oriented edge of $D$ with direction from $v$ to $u$. We call an oriented edge $(v,u)$  $\sigma$-\emph{decreasing}  if $v>u$ and $\sigma(uv)=+1$, that is, $(v,u)$ is positive and oriented from the larger vertex to the smaller vertex.  We note that a negative edge will never be $\sigma$-decreasing, no matter how it is oriented.  An orientation $D$ of $G$ is called $\sigma$-\emph{even} if it has an even number of $\sigma$-decreasing edges  and called $\sigma$-\emph{odd} otherwise. For a nonnegative sequence $\bm{d}=(d_v)_{v\in V(G)}$, let $\sigma{}EO(\bm{d})$ and  ${\sigma}OO(\bm{d})$ denote the sets of all $\sigma$-even and $\sigma$-odd orientations of $G$ having outdegree sequence $\bm{d}$, respectively.
\begin{lem}\label{EO-OO}
$P_{G,\sigma}(\bm{x})=\sum_{}(|\sigma{}EO(\bm{d})|-|\sigma{}OO(\bm{d})|)\prod_{v\in V(G)}x_v^{d_v}$,
where $\bm{d}=(d_v)_{v\in V(G)}$ and the summation is taken over all $\bm{d}$ such that  $d_v\ge 0$ and $\sum _{v\in V(G)}d_v=|E(G)|$.
\end{lem}
\begin{proof}
Let $D$ be an arbitrary orientation of $G$. For each oriented edge $e=(v,u)$, define
\begin{equation}
    w(e)=
   \begin{cases}
   -x_v, &\mbox{if $e$ is $\sigma$-decreasing}\\
   x_v, &\mbox{otherwise.}
   \end{cases}
  \end{equation}
  and $w(D)=\prod_{e\in E(D)} w(e)$. Let $d_v$ be the outdegree of $v$ in $D$ for each $v\in V(G)$  and let $t$ be the number of $\sigma$-decreasing edges in $D$. It is easy to see that
  \begin{equation}\label{wD}
  w(D)=(-1)^t\prod_{v\in V(G)}x_v^{d_v}.
  \end{equation}
  Recall that
  $$P_{G,\sigma}(\bm{x})=\prod_{u\sim v,u<v}(x_u-\sigma(uv)x_v).$$
  By selecting $x_u$ or $-\sigma(uv)x_v$ from each factor $(x_u-\sigma(uv)x_v)$, we expand $P_{G,\sigma}(\bm{x})$ and obtain $2^{|E(G)|}$ monomials, each of which has coefficient $\pm 1$.  For each monomial, we orient the
edge $uv$ of $G$ with direction from $u$ to $v$ if, in the factor $(x_u-\sigma(uv)x_v)$,  $x_u$ is selected; or  from $v$ to $u$ if  $-\sigma(uv)x_v$ is selected. This is clearly a bijection between the $2^{|E(G)|}$ monomials and the $2^{|E(G)|}$ orientations of $G$. Therefore,
  \begin{equation}\label{pGwD}
  P_{G,\sigma}(\bm{x})=\sum w(D),
  \end{equation} where $D$ ranges over all orientations of $G$.

  Let $\bm{d}=(d_v)_{v\in V(G)}$ be the sequence of outdegrees of some orientation $D$. Clearly, $d_v\ge 0$ and $\sum_{v\in V(G)} d_v=|E(G)|$. Note that there are exactly $|\sigma{}EO(\bm{d})|$ (resp. $|\sigma{}OO(\bm{d})|$) $\sigma$-even (resp. $\sigma$-odd) orientations of $G$. It follows from (\ref{wD}) and (\ref{pGwD}) that the coefficient of
  $\prod_{v\in V(G)}x_v^{d_v}$ in the expansion of $P_{G,\sigma}(\bm{x})$  is $|\sigma{}EO(\bm{d})|-|\sigma{}OO(\bm{d})|$. This proves the lemma.
\end{proof}
For an orientation $D$ of $G$, a subdigraph $H$ of $D$ is called \emph{Eulerian} if $V(H)=V(D)$ and the indegree of every vertex equals its outdegree. We note that an Eulerian subdigraph $H$ defined here is not necessarily connected. In particular, a vertex is called {\it isolated}  in $H$ if it has indegree 0 (and therefore, has outdegree 0) in $H$. Further, $H$ is called $\sigma$-\emph{even} (resp. $\sigma$-\emph{odd}) if $H$ has an even (resp. odd) number of positive edges.  Let ${\sigma}EE(D)$ (resp. ${\sigma}OE(D)$) denote the set of all $\sigma$-even (resp. $\sigma$-odd) Eulerian subdigraphs of $D$.

\begin{lem}\label{EE-OE}
Let $(G,\sigma)$ be a signed graph and $D$ be an orientation of $G$ with outdegree sequence $\bm{d}=(d_v)_{v\in V(G)}$. Then the coefficient of $\prod_{v\in V(G)}x_v^{d_v}$ in the expansion of $P_{G,\sigma}(\bm{x})$ is equal to $\pm (|\sigma{}EE(D)|-|\sigma{}OE(D)|)$.
\end{lem}
\begin{proof}
For any orientation $D'\in \sigma{}EO(\bm{d})\cup \sigma{}OO(\bm{d})$, let $D\oplus D'$ denote the set of all oriented edges of $D$ whose orientation in $D'$ is in the opposite direction. Since  $D$ and $D'$ have the same outdegree sequence, $D\oplus D'$ is Eulerian. Moreover, $D\oplus D'$ contains an even number of positive edges if and only if $D$ and $D'$ are both $\sigma$-even or both $\sigma$-odd.

Now, the map $\tau\colon\,D'\mapsto D\oplus D'$ is clearly a bijection between  $\sigma{}EO(\bm{d})\cup \sigma{}OO(\bm{d})$ and  $\sigma{}EE(D)\cup \sigma{}OE(D)$. If $D$ is $\sigma$-even, then $\tau$ maps $\sigma{}EO(\bm{d})$ to $\sigma{}EE(D)$ and maps $\sigma{}OO(\bm{d})$ to $\sigma{}OE(D)$. In this case $|\sigma{}EO(\bm{d})|=|\sigma{}EE(D)|$ and $|\sigma{}OO(\bm{d})|= |\sigma{}OE(D)|$.  Thus, $|\sigma{}EO(d)|-|\sigma{}OO(D)|=|\sigma{}EE(D)|-|\sigma{}OE(D)|$. It follows from Lemma \ref{EO-OO} that the coefficient of $\prod_{v\in V(G)}x_v^{d_v}$ in the expansion of $P_{G,\sigma}(\bm{x})$ is equal to $|\sigma{}EE(D)|-|\sigma{}OE(D)|$. Similarly, if $D$ is $\sigma$-odd, then the coefficient of $\prod_{v\in V(G)}x_v^{d_v}$ in the expansion of $P_{G,\sigma}(\bm{x})$ is equal to $|\sigma{}OE(D)|-|\sigma{}EE(D)|$. This proves the lemma.
\end{proof}
By Lemma \ref{EE-OE} and Definition \ref{AT}, we have the following characterization of the Alon-Tarsi number $AT(G,\sigma)$.
\begin{cor}\label{ATresult} For any signed graph $(G,\sigma)$,  $AT(G,\sigma)$ equals the minimum $k$ for which there exists an orientation $D$ of $G$  such that $|\sigma{}EE(D)|\neq |\sigma{}OE(D)|$ and every vertex has outdegree less than $k$.
\end{cor}
\section{Proof of Theorem \ref{AT5}}
We call a plane graph (a planar graph embedded on the plane) a \emph{near triangulation} if   the boundary of the outer face is a cycle, called  the outer facial cycle, and the boundaries of all inner faces are triangles.
\begin{defi}\label{nice}\textup{
Let $(G,\sigma)$ be a signed graph where $G$ is a near triangulation with outer  facial cycle $v_1v_2\cdots v_k$ and let $e=v_1v_2$. An orientation $D$ of $G-e$ is $\sigma$-\emph{nice} for $G-e$ if the following hold:
\begin{itemize}
\item{$|\sigma{}EE(D)|\neq|\sigma{}OE(D)|$.}
\item{ $v_1$ and $v_2$ have outdegree 0, $v_i$ has outdegree at most 2 for $i\in\{3,4,\ldots,k\}$, and every interior vertex has outdegree at most 4.}
\end{itemize}
}
 \end{defi}
We use the method presented in \cite{Zhu2018} to prove the following theorem.
\begin{thm}\label{key}
Let $(G,\sigma)$ be a signed graph where $G$ is a near triangulation with outer  facial cycle $C=v_1v_2\cdots v_k$ and let $e=v_1v_2$. Then $G-e$ has a $\sigma$-nice orientation.
\end{thm}
\begin{proof}
We prove the theorem by induction on $|V(G)|$. If $|V(G)|=3$ then $G-e$ is a path $v_2v_3v_1$.  Let $D$ be the orientation of $G-e$ such that $E(D)=\{(v_3,v_2),(v_3,v_1)\}$.  Clearly, $D$ is $\sigma$-nice. Now assume that $|V(G)| > 3$ and the assertion holds for graphs of order less than $|V(G)|$.
We shall distinguish two cases, according to whether  the outer facial cycle $C$ contains a chord incident with $v_k$.

First we consider the case that $C$ has a chord $e'=v_kv_j$ where $2\le j\le k-2$ (see Figure 1(a)). In this case  $C_1=v_1v_2\cdots v_jv_k$ and $C_2=v_kv_jv_{j+1}\cdots v_{k-1}$ are two cycles of $G$.   For $i\in\{1,2\}$, let $G_i$ be the subgraph of $G$ formed by $C_i$ and its interior part.  By the induction hypothesis, $G_1-e$ has a $\sigma$-nice orientation $D_1$, and $G_2-e'$ has a $\sigma$-nice orientation $D_2$.  We notice that $D_1$ and $D_2$ are edge disjoint. Let $D=D_1\cup D_2$.  It is clear that $D$ is an orientation of $G-e$. We will show that $D$ is $\sigma$-nice for $G-e$.  It can be checked that $D$ satisfies the outdegree condition in Definition \ref{nice}. It remains to check that  $|\sigma{}EE(D)|\neq|\sigma{}OE(D)|$.

  Note that both $v_k$ and $v_j$ have outdegree 0 in $D_2$.  This implies that $v_k$ and $v_j$ are both isolated in any Eulerian subdigraph of $D$. Therefore, any Eulerian subdigraph $H$ of $D$ has an edge-disjoint decomposition $H=H_1\cup H_2$, where $H_1$ and $H_2$ are Eulerian subdigraphs in $D_1$ and $D_2$, respectively. Thus, the map $\tau\colon\,H\mapsto (H_1,H_2)$ is a bijection between $\sigma{}EE(D)\cup \sigma{}OE(D)$ and $(\sigma{}EE(D_1)\cup \sigma{}OE(D_1))\times (\sigma{}EE(D_2)\cup \sigma{}OE(D_2))$.  Moreover, $H$ is $\sigma$-even if and only if both $H_1$ and $H_2$ are $\sigma$-even, or both are $\sigma$-odd. Thus, we have
\begin{eqnarray*}\label{diff}
&&|\sigma{}EE(D)|-|\sigma{}OE(D)|\\
&=&(|\sigma{}EE(D_1)\times \sigma{}EE(D_2)|+|\sigma{}OE(D_1)\times \sigma{}OE(D_2)|)\\
&&-(|\sigma{}EE(D_1)\times \sigma{}OE(D_2)|+|\sigma{}OE(D_1)\times \sigma{}EE(D_2)|)\\
&=&(|\sigma{}EE(D_1)|-|\sigma{}OE(D_1)|)\cdot(|\sigma{}EE(D_2)|-|\sigma{}OE(D_2)|)\\
&\ne&0,
  \end{eqnarray*}
where the last inequality holds since $D_1$ and $D_2$ are $\sigma$-nice. This proves that $D$ is a $\sigma$-nice orientation of $G-e$.

Next assume that $C$ contains no chord of the form $v_kv_j$ for $j\in \{2,3,\ldots, k-2\}$.   Let $v_{k-1}$,$u_1$,$u_2,\ldots,u_s,v_{1}$ be the neighbors of $v_k$ and be ordered  so that $v_kv_{k-1}u_1$, $v_ku_1u_2$,\ldots,$v_ku_sv_{1}$ are inner facial cycles of $G$ (see Figure 1(b) when $k=3$ and  Figure 1(c) when $k>3$).  Let $G'=G-v_k$. It is clear that $G'$ is a near triangulation with outer facial cycle $v_1v_2\cdots v_{k-1}$$u_1u_2\cdots u_s$. Therefore, by the induction hypothesis, $G'-e$ has a $\sigma$-nice orientation $D'$.
\begin{figure}[htbp]
\begin{center}
\includegraphics[height=10cm]{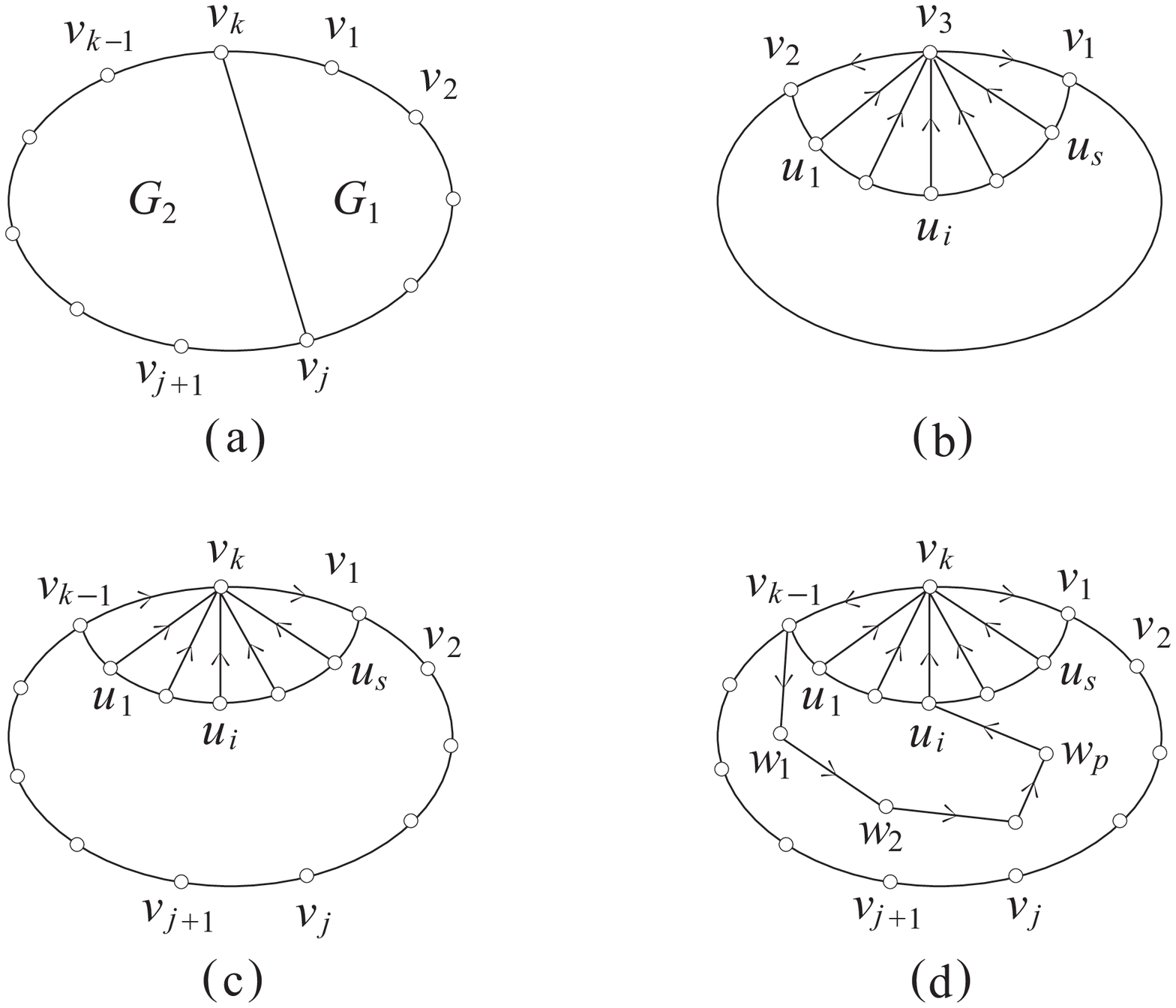}

{\bf Figure 1.} Proof of Theorem \ref{AT5}.
\end{center}
\end{figure}

If $k=3$ (i.e., $C$ is a triangle), then let $D$ be the orientation of $G-e$ obtained from $D'$ by adding the vertex $v_3$ and oriented edges $(v_3,v_1)$, $(v_3,v_2)$ and $(u_i,v_3)$ for $i\in \{1,2,\ldots,s\}$, as shown in Figure 1(b). It is easy to verify that $D$ satisfies the outdegree condition in Definition \ref{nice}. In particular, both $v_1$ and $v_2$ have outdegree 0. Thus,   $v_1$ and $v_2$ are  both isolated in any Eulerian subdigraph of $D$ and therefore, by the definition of $D$, $v_3$ is also isolated in any Eulerian subdigraph of $D$. This means that each Eulerian subdigraph of $D$ is an Eulerian subdigraph of $D'$ by  ignoring the isolated vertex $v_k$. Thus, $\sigma{}EE(D)=\sigma{}EE(D')$ and $\sigma{}OE(D)=\sigma{}EE(D')$. As $D'$ is $\sigma$-nice, $|\sigma{}EE(D')|\neq |\sigma{}OE(D')|$ and hence $|\sigma{}EE(D)|\neq |\sigma{}OE(D)|$. This proves that $D$ is a $\sigma$-nice orientation of $G-e$.

Now assume that $k\ge 4$. We call an orientation $D$ of $G'-e$ \emph{special} if the following hold:
\begin{itemize}
\item{$v_1$ and $v_2$ have outdegree 0, $v_{k-1}$ has outdegree at most 1, each of $v_3,v_4,\ldots,v_{k-1}$ has outdegree at most 2, and each of $u_1,u_2,\ldots,u_s$ has outdegree at most 3.}
\item{Every interior vertex has outdegree at most 4.}
\end{itemize}
To show that $G-e$ has a $\sigma$-nice orientation, we consider two cases:

\noindent\textbf{Case 1.} $G'-e$ has a special orientation $D''$ with $|\sigma{}EE(D'')|\neq|\sigma{}OE(D'')|$.

Let $D$ be the orientation of $G-e$ obtained from $D''$ by adding the vertex $v_k$ and $s+2$ oriented edges $(v_k,v_1)$,$(v_{k-1},v_k)$ and $(u_i,v_k)$ for $i\in \{1,2,\ldots,s\}$, see Figure 1(c). Then $D$ satisfies the outdegree condition of a $\sigma$-nice orientation. Since $v_1$ has outdegree 0 in $D$,  by a similar discussion as above, $v_k$ is isolated in any Eulerian subdigraph of $D$. Therefore, each Eulerian subdigraph of $D$ is an Eulerian subdigraph of $D''$ by ignoring the isolated vertex $v_k$, i.e., $\sigma{}EE(D)=\sigma{}EE(D'')$ and $\sigma{}OE(D)=\sigma{}OE(D'')$. This yields that $|\sigma{}EE(D)|\neq |\sigma{}EE(D)|$ by the condition of this case. Thus, $D$ is a $\sigma$-nice orientation of $G-e$, as desired.

\noindent\textbf{Case 2.} For any special orientation $D''$ (if exists), $|\sigma{}EE(D'')|=|\sigma{}OE(D'')|$.

Recall that $D'$ is a $\sigma$-nice orientation of $G'-e$. Let $D$ be the orientation of $G-e$ obtained from $D'$ by adding the vertex $v_k$ and $s+2$ oriented edges $(v_k,v_1)$, $(v_k,v_{k-1})$ and $(u_i,v_k)$ for $i\in \{1,2,\ldots,s\}$, as shown in Figure 1(d). Clearly, $D$ satisfies the outdegree condition of a $\sigma$-nice orientation. To show that $D$ is $\sigma$-nice for $G-e$, it remains to show that $|\sigma{}EE(D)|\neq |\sigma{}OE(D)|$.

Notice that  $v_1$ has outdegree 0 in $D$ and therefore,  is isolated in any Eulerian subdigraph of $D$. Thus, if $H$ is an Eulerian subdigraph of $D$  and $v_k$ is non-isolated in $H$ then $H$ contains the oriented edge $(v_k,v_{k-1})$ and exactly one of the $s$ oriented edges $(u_1,v_k)$, $(u_2,v_k),\ldots,(u_s,v_k)$. For $i\in\{1,2,\ldots,s\}$, let
$$\sigma{}EE_i(D)=\{H\in\sigma{}EE(D)\colon\,(u_i,v_k)\in H\},$$
and similarly,
$$\sigma{}OE_i(D)=\{H\in \sigma{}OE(D)\colon\,(u_i,v_k)\in H\}.$$
For an Eulerian subdigraph of $D'$, we regard it as an Eulerian subdigraph of $D$ by adding $v_k$ as an isolated vertex. Then we have
$$\sigma{}EE(D)=\sigma{}EE(D')\cup \bigcup_{i=1}^{s}\sigma{}EE_i(D),\sigma{}OE(D)=\sigma{}OE(D')\cup \bigcup_{i=1}^{s}\sigma{}OE_i(D).$$
Since $D'$ is $\sigma$-nice, $|\sigma{}EE(D')|\neq |\sigma{}OE(D')|$. If we can show that $|\sigma{}EE_i(D)|=|\sigma{}OE_i(D)|$ for each $i\in \{1,2,\ldots,s\}$, then $|\sigma{}EE(D)|\neq|\sigma{}OE(D)|$ and we are done.

Let $i$ be any integer in $\{1,2,\ldots,s\}$. If $\sigma{}EE_i(D)\cup\sigma{}OE_i(D)=\emptyset$ then $|\sigma{}EE_i(D)|=|\sigma{}OE_i(D)|=0$, as desired. Thus, we may assume that $\sigma{}EE_i(D)\cup\sigma{}OE_i(D)\neq \emptyset$. Therefore, $D$ has an Eulerian subdigraph and hence a directed cycle containing $(u_i,v_k)$. Let $C_i=u_iv_kv_{k-1}w_1w_2\cdots w_p$ be such a directed cycle and let $D'_i$ be the orientation of $G'-e$ obtained from $D'$ by reversing the direction of edges in the path $v_{k-1}w_1w_2\cdots w_pu_i$. The reversing operation decreases the outdegree of $v_{k-1}$ by 1, increases the outdegree of $u_i$ by 1, and leaves the outdegrees of other vertices in $G'-e$ unchanged. Since $D'$ is $\sigma$-nice for $G'-e$, the outdegree condition of $D'$ implies that $D'_i$ is special. Hence,
$|\sigma{}EE(D'_i)|=|\sigma{}OE(D'_i)|$ by the condition of this case.

Let $C_i^{-1}$ be the reverse of $C_i$, i.e., $C_i^{-1}=w_pw_{p-1}\cdots w_1v_{k-1}v_ku_i$. For each Eulerian subdigraph $H\in \sigma{}EE_i(D)\cup\sigma{}OE_i(D)$, let $H\bigtriangleup C_i^{-1}$ be the symmetry difference of the edge sets of $H$ and $C_i^{-1}$, that is, the set obtained from the edge union $H\cup C_i^{-1}$ of $H$ and $C_i^{-1}$ by deleting the directed $2$-cycles. One may verify that  $H \bigtriangleup C_i^{-1}$ is an Eulerian subdigraph of $D'_i$ and the map $\tau\colon\, H\mapsto H\bigtriangleup C_i^{-1}$ is a bijection between $\sigma{}EE_i(D)\cup\sigma{}OE_i(D)$ and $\sigma{}EE(D'_i)\cup\sigma{}OE(D'_i)$.

For a set $S$ of some oriented edges in an orientation of $(G,\sigma)$,  we use $N^\sigma(S)$ to denote the number of positive edges in $S$.  If $S$ is a directed $2$-cycle, then either $N^\sigma(S)=2$ or $N^\sigma(S)=0$. Thus, $N^\sigma(H\bigtriangleup C_i^{-1})$ and $N^\sigma(H\cup C_i^{-1})$ have the same parity. Of course, $N^\sigma(H\cup C_i^{-1})=N^\sigma(H)+N^\sigma(C_i^{-1})=N^\sigma(H)+N^\sigma(C_i)$. Therefore, if $N^\sigma(C_i)$ is even, then $\tau\colon\, H\mapsto H\bigtriangleup C_i^{-1}$ maps $\sigma{}EE_i(D)$ to $\sigma{}EE(D'_i)$ and $\sigma{}OE_i(D)$ to $\sigma{}OE(D'_i)$. Similarly,  if $N^\sigma(C_i)$ is odd, then it  maps $\sigma{}EE_i(D)$ to $\sigma{}OE(D'_i)$ and $\sigma{}OE_i(D)$ to $\sigma{}EE(D'_i)$. Therefore,  we have
$|\sigma{}EE_i(D)|-|\sigma{}OE_i(D)|=\pm(|\sigma{}EE(D'_i)|-|\sigma{}OE(D'_i)|)$. Note that $D'_i$ is special. It follows from the condition of this case that $|\sigma{}EE_i(D)|=|\sigma{}OE_i(D)|$. This completes the proof of this theorem.
\end{proof}
\noindent\emph{Proof of  Theorem \ref{AT5}} Note that $AT(H,\sigma)\le AT(G,\sigma)$ for any subgraph $H$ of $G$. To show that $AT(G,\sigma)\le 5$ for any planar graph, it suffices to consider the case when $G$ is a near triangulation. Let $v_1v_2\cdots v_k$ be the outer facial cycle of $G$ and $e=v_1v_2$. By Theorem \ref{key}, $G-e$ has a $\sigma$-nice orientation $D$. Let $D'$ be obtained from $D$ by adding the oriented edge $(v_1,v_2)$. Clearly, each vertex has outdegree at most $4$ in $D'$. Moreover, as $v_2$ has outdegree 0 in $D'$, the orientated edge $(v_1,v_2)$ will never appears in any Eulerian subgraph of $D'$. Thus, $|\sigma{}EE(D')|=|\sigma{}EE(D)|$ and  $|\sigma{}OE(D')|=|\sigma{}OE(D)|$. As $D$ is $\sigma$-nice, we have $|\sigma{}EE(D)|\neq |\sigma{}OE(D)|$. Therefore, $|\sigma{}EE(D')|\neq |\sigma{}OE(D')|$ and hence $AT(G,\sigma)\le 5$ by Corollary \ref{ATresult}.

\section{Proof of Theorem \ref{AT4}}
For a graph $G$, the maximum average degree of $G$, denoted $\textup{mad}(G)$, is the maximum of $2|E(H)|/|V(H)|$, where $H$ ranges over all subgraphs of $G$. The following useful criterion on the existence of an orientation with bounded outdegree appeared in \cite{Alon1992}.
\begin{lem}\label{madori}
A graph $G$ has an orientation $D$ such that every vertex has outdegree at most $p$ if and only if $\textup{mad}(G)\le 2p$.
\end{lem}
\begin{cor}\label{negative}
For any graph $G$,
\begin{equation}
AT(G,-)=\left\lceil\frac{\textup{mad}(G)}{2}\right\rceil+1.
\end{equation}
\end{cor}
\begin{proof}
Let $p=\lceil\frac{\textup{mad}(G)}{2}\rceil$. Then $\textup{mad}(G)\le 2p$ and hence, by Lemma \ref{madori}, $G$ has an orientation $D$ in which every outdegree is at most $p$. As all edge in $(G,-)$ is negative, each Eulerian subdigraph of $D$ contains no positive edge and hence is $\sigma{}$-even. Thus $|\sigma{}OE(D)|=0$. Since the empty subdigraph is a $\sigma$-even Eulerian subdigraph, we have $|\sigma{}EE(D)|\ge 1$ and hence $|\sigma{}EE(D)|\neq |\sigma{}OE(D)|$. Thus by Corollary \ref{ATresult}, $AT(G,-)\le p+1.$

On the other hand, by Corollary \ref{ATresult}, $G$ has an orientation $D$ such that each outdegree is at most $AT(G,-)-1$. Thus, by Lemma \ref{madori}, $\textup{mad}(G)\le 2(AT(G,-)-1)$, i.e.,
$AT(G,-)\ge \frac{\textup{mad}(G)}{2}+1$. Therefore, $AT(G,-)\ge p+1$ since $AT(G,-)$ is an integer. This proves the corollary.
\end{proof}
\emph{Proof of  Theorem \ref{AT4}}. For a signed graph $(G,\sigma)$, Schweser and Stiebitz \cite{Schweser2017} showed that $\chi(G,\sigma)\le 2$ if and only if $(G,\sigma)$ is switching equivalent to $(G,-)$. Thus, by Proposition \ref{equAT}, it suffices to consider the case when $(G,\sigma)=(G,-)$, i.e., $\sigma(uv)=-1$ for each $uv\in E(G)$. Let $H$ be any subgraph of a planar graph $G$. Then by Euler's formula for planar graph we have  $2|E(H)|/|V(H)|\le 6$, i.e., $\textup{mad}(G)\le 6$. By  Corollary \ref{negative}, $AT(G,-)\le 4$. This proves the first part of Theorem \ref{AT4}.

Let $(G,-)$ be the negative planar graph as shown in Figure 2. We show that $(G,-)$ is not $3$-choosable.
\begin{figure}[htbp]
\begin{center}
\includegraphics[height=4.2cm]{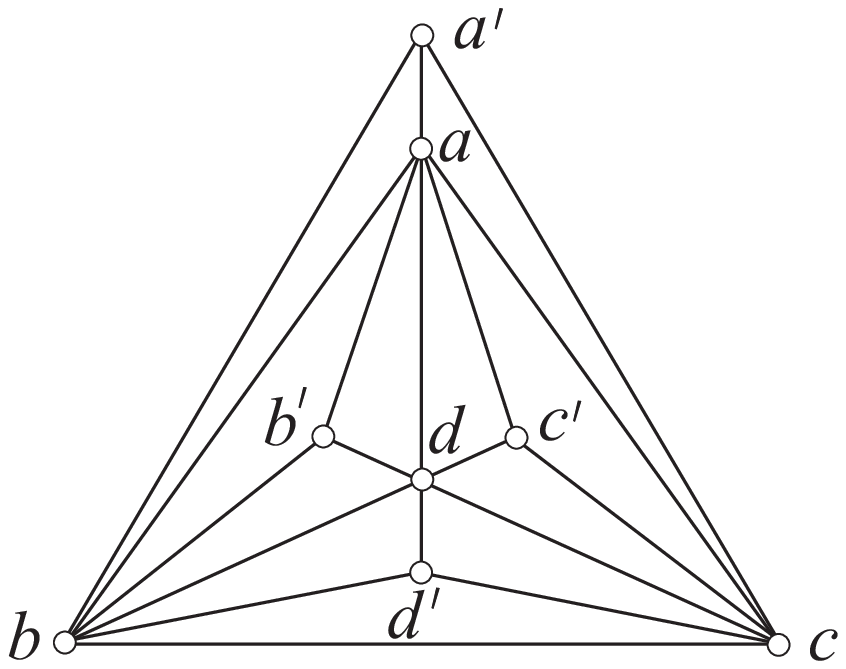}

{\bf Figure 2.} A non-3-choosable negative planar graph $(G,-)$.
\end{center}
\end{figure}\\
Define a $3$-list assignment $L$ as follows:
\begin{itemize}
\item{$L(a)=L(a')=\{0,-1,-2\}$.}
\item{ $L(b)=L(b')=\{0,-1,2\}$.}
\item{$L(c)=L(c')=\{0,1,-2\}$.}
\item{$L(d)=L(d')=\{0,1,2\}$.}
\end{itemize}
It suffices to show that $(G,-)$ is not $L$-colorable.  Suppose to the contrary that $\phi$ is an $L$-coloring of $(G,-)$.  Let $V=\{a,b,c,d\}$.

\noindent\emph{Claim 1}: There exists some  $x\in V$ such that $\phi(x)=0$.

Suppose to the contrary that $\phi(x)\neq 0$ for each $x\in V$. Then $\phi(a)\in \{-1,-2\}$, $\phi(b)\in \{-1,2\}$, $\phi(c)\in\{1,-2\}$ and $\phi(d)\in\{1,2\}$. Note that $(G[V],-)$ is a negative complete graph. Thus $\phi(x)\neq -\phi(y)$ for two distinct $x$, $y$ in $V$. If $\phi(a)=-1$ then $\phi(c)=-2$ and $\phi(d)=2$. Now, $\phi(c) =-\phi(d)$, a contradiction. Similarly, if $\phi(a)=-2$ then $\phi(b)=-1$ and $\phi(d)=1$ and hence $\phi(b)=-\phi(d)$. This is also a contradiction. Thus, Claim 1 follows.

\noindent\emph{Claim 2}: Let $x\in V$. If $\phi(x)=0$  then $\phi(N(x'))=-L(x')$.

We only prove the case that $x=a$ and the other three cases can be settled in the same way.  Since $\phi(a)=0$, we have $\phi(b)\in \{-1,2\}$, $\phi(c)\in\{1,-2\}$ and $\phi(d)\in\{1,2\}$. If $\phi(b)=-1$ then $\phi(c)=-2$ and $\phi(d)=2$. Thus, $\phi(c)=-\phi(d)$, a contradiction. Therefore, $\phi(b)=2$. Similarly, if $\phi(c)=-2$ then $\phi(b)=-1$ and $\phi(d)=1$. We also have a contradiction as $\phi(b)=-\phi(d)$. Therefore, $\phi(c)=1$. Finally, as $N(a')=\{a,b,c\}$ and $L(a')=\{0,-1,-2\}$, we have $\phi(N(a'))=\{\phi(a),\phi(b),\phi(c)\}=\{0,2,1\}=-L(a')$. This proves Claim 2.

Now, by Claim 1, let $x\in V$ satisfy $\phi(x)=0$. Then, $\phi(N(x'))=-L(x')$ by Claim 2. As $\phi(x')\in L(x')$ we have $-\phi(x')\in \phi(N(x'))$, that is, $-\phi(x')=\phi(y)$ for some $y\in N(x')$. Thus, $\phi$ is not proper since $x'y$ is a negative edge. This is a contradiction and hence completes the proof of Theorem \ref{AT4}.


\end{document}